\documentclass[12pt]{amsart}
\usepackage{amsfonts, amssymb, amsmath, amsthm}
\usepackage[mathscr]{eucal}
\usepackage[hmargin=1.2in, vmargin=1.5in]{geometry}
\input amssym.def
\input amssym

\newtheorem{theorem}{Theorem}[section]
\newtheorem{proposition}[theorem]{Proposition}

\theoremstyle{definition}
\newtheorem{definition} [theorem]{Definition}

\theoremstyle{remark}

\numberwithin{equation}{section}
\begin{document}

\title[Directional Stockwell transform of distributions]{Directional Stockwell transform of distributions}

\author[A. Ferizi]{Astrit Ferizi}
\address{Faculty of Mathematics and Natural Sciences, University of Prishtina, George Bush 31, Prishtina, 10000, Kosovo}
\email{ferizi.astrit@gmail.com}

\author[K. Saneva]{Katerina Hadzi-Velkova Saneva}
\address{Faculty of Electrical Engineering and Information Technologies, Ss. Cyril and Methodius University in Skopje, Rugjer Boshkovikj 18, Skopje, 1000, Republic of North Macedonia}
\email {saneva@feit.ukim.edu.mk}


\begin{abstract}
We introduce and study the directional Stockwell transform as a hybrid of the directional short-time Fourier transform and the ridgelet transform. We prove an extended Parseval identity and a reconstruction formula for this transform, as well as results for the continuity of both the directional Stockwell transform and its synthesis transform on the appropriate space of test functions. Additionally, we develop a distributional framework for the directional Stockwell transform on the Lizorkin space of distributions $\mathcal{S}_{0}'(\mathbb{R}^n)$.
\end{abstract}

\maketitle

\section{Introduction}

The one-dimensional wavelet transform has proven to be a very efficient tool for detecting the point singularity of a function. However, it does not perform as well if the singularity is well distributed, such as singularities along curves, along hyperplanes, etc. \cite{Can99}. One of the developed techniques for dealing with higher-dimensional phenomena is ridgelet analysis. 
The idea of the ridgelet transform is to use the Radon transform to project a hyperplane singularity into a point singularity and then act with the one-dimensional wavelet transform. It was introduced by Cand\`{e}s in \cite{CanPhd, Can999} who constructed from a Schwartz function $\psi$, a family of building blocks, called  ridgelets,
$$\textbf{x} \to \frac{1}{a} \psi \left( \frac{\textbf{u} \cdot \textbf{x}-b}{a} \right),  \enspace \textbf{x}\in \mathbb{R}^n,$$
where $\textbf{u}\in \mathbb{S}^{n-1}$ is the orientation parameter ($\mathbb{S}^{n-1}$ is the unit sphere of $\mathbb{R}^n$), $b \in \mathbb{R}$  is the location
parameter and $a \in \mathbb{R}^{+}$  is the scale parameter. Then, he used them to define the ridgelet transform of $f\in L^1(\mathbb{R}^n)$,
\begin{equation}\label{RT}{\mathcal R}_{\psi} f (\textbf{u},b,a)=  \int_{\mathbb{R}^n}f(\textbf{x}) \overline{\psi}\left( \frac{\textbf{u} \cdot \textbf{x}-b}{a} \right) d\textbf{x}.\end{equation}
The ridgelet transform is extended and studied on the space of Lizorkin distributions $\mathcal{S}'_0(\mathbb{R}^n)$ in \cite{Kos13}. Since the ridgelets are not from the class $\mathcal{S}_0(\mathbb{R}^n)$, the ridgelet transform of Lizorkin distributions can not be evaluated directly as an act of distribution on the ridgelets. Therefore the authors of \cite{Kos13} defined the ridgelet transform on $\mathcal{S}_0(\mathbb{R}^n)$  via a duality approach. They also provided a relationship between the ridgelet, Radon, and  wavelet transforms.

Similarly to the ridgelet analysis,  Grafakos and Sansing provided an idea to localize information in time, frequency and direction. They constructed the Gabor ridge functions of a Schwartz function $\psi$,
$$\textbf{x} \to \psi(\textbf{u} \cdot \textbf{x}-b)e^{ia(\textbf{x} \cdot \textbf{u})},  \enspace \textbf{x}\in \mathbb{R}^n,$$
where $\textbf{u}\in \mathbb{S}^{n-1}$ and $a,b \in \mathbb{R}$ \cite{Gra08} and used them to define a directional sensitive variant of the short-time Fourier transform of $f\in L^1(\mathbb{R}^n)$,
\begin{equation}\label{DSTFT} \frac{1}{\sqrt {2\pi}}\int_{\mathbb{R}^n}f(\textbf{x}) \overline{\psi}(\textbf{u} \cdot \textbf{x}-b)e^{-ia(\textbf{x} \cdot \textbf{u})} d\textbf{x}.\end{equation} Their results
for directionally sensitive time–frequency decompositions in $L^2(\mathbb{R}^n)$ based on Gabor
systems in $L^2(\mathbb{R})$ are generalized in \cite{Christensen}, by showing similar results for general frames for $L^2(\mathbb{R})$, both in the setting of discrete and continuous frames.

The Stockwell transform was introduced by Stockwell in \cite{Sto96} and combines the best features of the short-time Fourier transform and the wavelet transform. It can be seen as the phase correction of the wavelet transform  \cite{Sto2007}. A generalized version of the Stockwell transform was given by Du et al. in \cite{Du07}, whereas Riba and Wong introduced and studied the multi-dimensional Stockwell transform \cite{Riba14}. We point out that the Stockwell transform is applied in various fields, including  the theory of distributions (see, e.g., \cite{Atanasova, San20}), signal and image processing, medicine, mechanical engineering, geoscience (see, e.g., \cite{Wei, Kumar} and references therein).

Given the limitations of the Stockwell analysis  within the time-frequency plane, we introduce the directional sensitive variant of the Stockwell transform, motivated by the ridgelet transform (\ref{RT})  and the transform (\ref{DSTFT}) defined by Grafakos and Sansing. This transform, called the directional Stockwell transform, aims to provide a more comprehensive representation of signals with directional features or anisotropic behaviors. 
We prove an extended Parseval identity, a reconstruction formula, and the continuity of the directional Stockwell transform and its transpose, called here the directional Stockwell synthesis operator, respectively, on the spaces $\mathcal{S}_{0}(\mathbb{R}^n)$ and $\mathcal{S}(\mathbb{Y}^{n+1})$, where $\mathbb{Y}^{n+1}=\mathbb{S}^{n-1} \times \mathbb{R} \times (\mathbb{R} \setminus \lbrace 0 \rbrace)$. 

We then use our results to develop a distributional framework for the directional Stockwell transform. In Section 5, following the duality approach, we define the directional Stockwell transform and its correspondent synthesis operator on $\mathcal{S}'_{0}(\mathbb{R}^n)$ and $\mathcal{S}'(\mathbb{Y}^{n+1})$, respectively. Here $\mathcal{S}'(\mathbb{Y}^{n+1})$ is a certain space of distributions of slow growth on $\mathbb{Y}^{n+1}$ and $\mathcal{S}'_{0}(\mathbb{R}^n)$ stands for the Lizorkin
distribution space (cf. Subsection 2.1). We emphasize that the Lizorkin spaces play a key role in Holschneider's approach to the wavelet transform of distributions \cite{Hol95}, as well as in the development of the distributional theory for the ridgelet and Stockwell transforms  \cite{Kos13, San20}. Many important
Schwartz distribution spaces, such as $\mathcal{E}'(\mathbb{R}^n)$, $\mathcal{O}'_{C}(\mathbb{R}^n)$, $L^p(\mathbb{R}^n)$, or the $\mathcal{D}'_{L^p}(\mathbb{R}^n)$ spaces, are embedded into $\mathcal{S}'_{0}(\mathbb{R}^n)$.

Unlike the Stockwell transform of a distribution \cite{San20}, the directional Stockwell transform can not be defined by a direct evaluation of the distribution at a certain family of test functions. The larger distribution space where the direct approach works is $D_{L^1}'(\mathbb{R}^n)$ (see Section 6 below).

%%%%%%%%%%%%%%%%%%%%%%%%%%%%%%%%%%%%%%%%%%%%%%%%%%%%%%%%%%%%%%
\section{Preliminaries}
\subsection{Notations and spaces}
 
In this section we give some notations and spaces that are used throughout this paper. For $\textbf{x}=(x_1,x_2,...,x_n), \textbf{y}=(y_1,y_2,...,y_n) \in \mathbb{R}^n$ and $ \alpha=(\alpha_1,\alpha_2,...,\alpha_n) \in \mathbb{N}_{0}^{n},$ we use standard notations of \textit{n}-dimensional calculus,
$\textbf{x}^{\alpha}=x_{1}^{\alpha_1}x_{2}^{\alpha_2}\cdots x_{n}^{\alpha_n}$, 
$\partial_{\textbf{x}}^{\alpha}=\partial^{\alpha_1}_{x_1}\partial^{\alpha_2}_{x_2}\cdots \partial^{\alpha_n}_{x_n}=\frac{\partial^{|\alpha|}}{\partial x_{1}^{\alpha_1} \partial x_{2}^{\alpha_2} \cdots \partial x_{n}^{\alpha_n}}$, $|\alpha|=\alpha_1+\alpha_2+...+\alpha_n$, $\vert \textbf{x} \vert$ denotes the Euclidean norm and $\textbf{x}\cdot \textbf{y}= x_1 y_1+x_2 y_2+\cdots x_n y_n$ the scalar product of $\textbf{x}$ and $\textbf{y}$. We write $A\lesssim B$ when $A \leq C \cdot B $ for some positive constant $C$. The Fourier transform $\mathcal{F}$ of a function $f \in L^{1}(\mathbb{R}^n)$ is defined as 
$\mathcal{F}f(\boldsymbol{\xi})=\widehat{f}(\boldsymbol{\xi})=\int_{\mathbb{R}^n}f(\textbf{x})e^{-i\textbf{x}\cdot \boldsymbol{\xi}}d\textbf{x}, \enspace   \boldsymbol{\xi} \in \mathbb{R}^n,$
and it is extended to $ L^2(\mathbb{R}^n)$ as usual \cite{Hor83}. 

We use the notation $\left( f, \varphi \right)_{L^2}$ for the $L^2$ inner product of $f$ and $\varphi$ and $\langle f, \varphi \rangle$ for the dual pairing between a distribution $f$ and a test function $\varphi$.
All dual spaces in the paper are equipped with the strong dual topology \cite{Tre67}.

The Schwartz space  $\mathcal{S}(\mathbb{R}^n)$ consists of all functions $\varphi\in C^{\infty}(\mathbb{R}^n)$ such that 
\begin{equation} \label{1}
\rho_m(\varphi)=\sup_{\textbf{x}\in \mathbb{R}^n, \,  |\alpha| \leq m } (1+|\textbf{x}|)^{m} | \partial^{\alpha}_\textbf{x}\varphi(\textbf{x}) | <\infty,
\end{equation}
for all $m\in \mathbb{N}_{0}$ \cite{Sch66, Hor83}. It is topologized by means of seminorms (\ref{1}). Its dual $\mathcal{S}^{'}(\mathbb{R}^n)$ is the space of tempered distributions. The Lizorkin space of test functions $\mathcal{S}_0(\mathbb{R}^n)$ consists of all $\varphi \in  \mathcal{S}(\mathbb{R}^n)$ such that $\int_{\mathbb{R}^n} \textbf{x}^m \varphi(\textbf{x})  d\textbf{x}=0$, for all $m\in \mathbb{N}_0^{n}$. It is provided with the topology inherited from $\mathcal{S}(\mathbb{R}^n)$ and is of crucial importance for our analysis. Its dual space $\mathcal{S}_0'(\mathbb{R}^n)$ is known as the space of Lizorkin distributions and it is canonically isomorphic to the quotient of $\mathcal{S}'(\mathbb{R}^n)$ by the space of polynomials \cite{Hol95}. 

 The space $D_{L^p}(\mathbb{R}^n)$, $1 \leq p \leq \infty$, consists of all smooth functions $\varphi$ such that all derivatives belong to 
$L^p(\mathbb{R}^n)$. Its dual space $D_{L^p}'(\mathbb{R}^n)$ consists of all distributions $f$ which can be represented as 
\begin{equation}\label{dlp}f=\sum_{j=1}^{N} \partial^{\alpha_j}{f_{j}},\, f_j \in L^p(\mathbb{R}^n)\end{equation} 
for some $N\in \mathbb{N}, \, \alpha_j \in \mathbb{N}_{0}^n$ (\cite{Sch66}, Thm. XXV, page 201).

We use the notation $\mathbb{Y}^{n+1}=\mathbb{S}^{n-1} \times \mathbb{R} \times \mathbb{R}^{\times}$, where $\mathbb{R}^{\times}=\mathbb{R} \setminus \lbrace 0 \rbrace$ and  $\mathbb{S}^{n-1}$ stands for the unit sphere of $\mathbb{R}^n$. We always suppose that $n \geq 2$. We introduce $\mathcal{S}(\mathbb{Y}^{n+1})$ as a space of all functions $\Phi \in C^{\infty}(\mathbb{Y}^{n+1})$ such that
\begin{equation}
\rho_{s,r}^{l,m,k}(\Phi)=\sup_{(\textbf{u},b,a)\in \mathbb{Y}^{n+1} }{(1+|b|^2)^{r/2} (|a|^s+|a|^{-s}) |\partial_{a}^{l}  \partial_{b}^{m} \Delta_{\textbf{u}}^{k} \Phi(\textbf{u},b,a)|}<\infty,
\end{equation}
for all $s,r,l,m,k \in \mathbb{N}_0$, where $\Delta_\textbf{u}$ stands for the Laplace-Beltrami operator on the unit sphere $\mathbb{S}^{n-1}$. It is topologized in the usually way.
Its dual space is denoted by $\mathcal{S}'(\mathbb{Y}^{n+1})$ and will be fundamental in our definition of the directional Stockwell transform of Lizorkin
distributions, as it contains the range of this transform (cf. Section 5). {We fix $|a|^{n-2}dbdad\textbf{u}$ as a standard measure on $\mathbb{Y}^{n+1}$, where $d\textbf{u}$ stands for the surface measure on the sphere $\mathbb{S}^{n-1}$.}
If $F$ is a locally integrable function of slow growth on $\mathbb{Y}^{n+1}$, i.e. if there exist $C>0, s\in \mathbb{N}_0$ such that 
$$|F(\textbf{u},b,a)| \leq C \left( |a|^s+|a|^{-s} \right) \left( 1+|b| \right)^s,  \enspace (\textbf{u},b,a)\in \mathbb{Y}^{n+1}, $$
then $F$ will be identified with an element of $\mathcal{S}'(\mathbb{Y}^{n+1})$ via the action 
\begin{equation} \label{standardidentification}
\langle F,\Phi \rangle:= \int_{\mathbb{S}^{n-1}}\int_{\mathbb{R}^{\times}} \int_{\mathbb{R}} F(\textbf{u},b,a)\Phi(\textbf{u},b,a)|a|^{n-2} dbdad\textbf{u},  \enspace \Phi \in \mathcal{S}(\mathbb{Y}^{n+1}).
\end{equation}

%%%%%%%%%%%%%%%%%%%%%%%%%%%%%%%%%%%%%%%%%%%%%%%%%%5
\subsection{One-dimensional Stockwell transform }

The Stockwell transform of $f\in L^2(\mathbb{R})$ with respect to a window $\psi\in L^1(\mathbb{R})\cap L^2(\mathbb{R})$ is defined as
 \begin{equation}\label{one-dimensiional ST}S_{\psi} f (b,a)= \frac{|a|}{\sqrt{2 \pi}} \int_{\mathbb{R}}f(x) \overline{\psi}\big(a(x-b)\big ) e^{-ixa} dx,\end{equation}
for  $b\in \mathbb{R}$ and $a\in \mathbb{R} \setminus \lbrace 0 \rbrace$ \cite{Sto96, Sto2007, Du07}. 
 The authors of \cite{San20} give a complete and rigorous distributional framework for the Stockwell transform based on the duality theory. Moreover, they showed that the distributional Stockwell transform can equivalently be defined by direct evaluation of a distribution on a certain family of test functions. So if $f\in \mathcal{S}'(\mathbb{R})$ and $\psi\in \mathcal{S}(\mathbb{R})$, one can replace (\ref{one-dimensiional ST}) by $$S_{\psi}f(b,a)=\frac{1}{\sqrt{2 \pi}} \Big\langle f({x}), |a|\,\overline{\psi}\big(a(x-b)\big ) e^{-ixa} \Big\rangle, b\in \mathbb{R},  a\in \mathbb{R} \setminus \lbrace 0 \rbrace.$$

%%%%%%%%%%%%%%%%%%%%%%%%%%%%%%%%%%%%%%%%%%%%%%%%%%%%%%%%%%%
\subsection{The Radon transform }

In this subsection we give some preliminaries of the Radon transform that are important for our analysis \cite{Her83, Hel99}. 

The Radon transform of $f \in L^{1}(\mathbb{R}^n)$ is given by 
$$
Rf_\textbf{u}(p)=\int_{\textbf{x} \cdot \textbf{u} =p} f(\textbf{x})d\textbf{x}= \int_{\mathbb{R}^n} f(\textbf{x}) \delta(p-\textbf{x}\cdot \textbf{u} )d\textbf{x},
$$
where $\textbf{u}\in \mathbb{S}^{n-1},\, p\in \mathbb{R}$ and $\delta$ is the Dirac-delta function \cite{Her83}. Using the Fubini's theorem, one can show that $Rf \in L^1(\mathbb{S}^{n-1} \times \mathbb{R})$ for $f\in L^1(\mathbb{R}^n)$. The dual Radon transform $R^*\varrho$ of a function $\varrho \in L^{\infty}(\mathbb{S}^{n-1}\times \mathbb{R})$ is defined as 
$$ R^*\varrho(\textbf{x})=\int_{\mathbb{S}^{n-1}} \varrho(\textbf{u},\textbf{x} \cdot \textbf{u})d\textbf{u},
$$ where $\textbf{x}\in \mathbb{R}^n$ \cite{Her83}.

A connection between the Fourier transform and the Radon transform, known as the Fourier slice theorem, states that for $f \in L^{1}(\mathbb{R}^n)$  is true \begin{equation}\label{Slice}\widehat{Rf_\textbf{u}}(p)=\widehat{f}(p\textbf{u}),\end{equation}
where $\textbf{u}\in \mathbb{S}^{n-1}$ and $p\in \mathbb{R}$ \cite{Her83}. 

%%%%%%%%%%%%%%%%%%%%%%%%%%%%%%%%%%%%%%%%%%%%%%%%%%%%%%%%
\section{Directional Stockwell transform of functions}

In this section we define the directional Stockwell transform and prove an extended Parseval identity and a reconstruction formula that suggests the definition of the directional Stockwell synthesis operator.

\begin{definition}
Let $\psi\in \mathcal{S}(\mathbb{R})$. The directional Stockwell transform of an integrable function $f\in L^1(\mathbb{R}^n)$ is defined as 
\begin{equation}\label{def_DST}
\begin{split}
DS_{\psi}f(\textbf{u},b,a):&=\left( f(\textbf{x}), \psi_{\textbf{u},b,a}(\textbf{x}) \right)_{L^2}\\
&= \frac{|a|}{(2 \pi)^{n/2}} \int_{\mathbb{R}^n}{f(\textbf{x})\overline{\psi}\big(a(\textbf{x} \cdot \textbf{u}-b)\big) e^{-ia(\textbf{x} \cdot \textbf{u})}d\textbf{x}},
\end{split}
\end{equation}
where $\psi_{\textbf{u},b,a}(\textbf{x})=\frac{|a|}{(2 \pi)^{n/2}} \psi(a(\textbf{x} \cdot \textbf{u}-b)) e^{ia(\textbf{x} \cdot \textbf{u})}, \ (\textbf{u},b,a)\in \mathbb{Y}^{n+1} , \ \textbf{x}\in \mathbb{R}^n$.
\end{definition}

The following proposition presents a useful relation between the directional Stockwell transform and the Fourier
transform.

\begin{proposition}
For $\psi\in \mathcal{S}(\mathbb{R})$ and $f\in L^1(\mathbb{R}^n)$ is true
\begin{equation*}DS_\psi f({\normalfont\textbf{u}},b,a)=\frac{e^{-iab}}{(2\pi)^{n/2+1}} \int_{\mathbb{R}}{\widehat{f}(\xi \normalfont\textbf{u}) \overline{\widehat{\psi}\left(\frac{\xi}{a}-1\right)}  e^{i \xi b} d \xi}.\end{equation*}
\end{proposition}
\begin{proof}
By the Fubini's theorem and the Fourier slice theorem (\ref{Slice}), we have
\begin{align*}
DS_\psi f(\textbf{u},b,a)&=  \frac{|a|}{(2 \pi)^{n/2}} \int_{\mathbb{R}^n}{f(\textbf{x})\overline{\psi(a(\textbf{x} \cdot \textbf{u}-b))} e^{-ia(\textbf{x} \cdot \textbf{u})}d\textbf{x}}\\
&=\frac{|a|}{(2 \pi)^{n/2}} \int_{\mathbb{R}^n}{f(\textbf{x})d\textbf{x} \int_{\mathbb{R}} \overline{\psi(a(p-b))} e^{-iap} \delta (p-\textbf{x} \cdot \textbf{u})dp }\\
&=\frac{|a|}{(2 \pi)^{n/2}} \int_{\mathbb{R}}{\overline{\psi(a(p-b))} e^{-iap} dp \int_{\mathbb{R}^n}{f(\textbf{x})\delta (p-\textbf{x} \cdot \textbf{u})d\textbf{x}}}\\
&=\frac{|a|}{(2 \pi)^{n/2}} \int_{\mathbb{R}}{Rf_{\textbf{u}}(p) \overline{\psi(a(p-b))} e^{-iap} dp }\\
&=\frac{1}{(2 \pi)^{n/2+1}} \int_{\mathbb{R}}{\widehat{Rf_{\textbf{u}}}(\xi) \overline{\widehat{\psi}\left(\frac{\xi}{a}-1\right)} e^{i(\xi-a)b} d \xi }\\
&=\frac{e^{-iab}}{(2\pi)^{n/2+1}} \int_{\mathbb{R}}{\widehat{f}(\xi \textbf{u}) \overline{\widehat{\psi}\left(\frac{\xi}{a}-1\right)}  e^{i \xi b} d \xi}.
\end{align*}
\end{proof}

From the previous proof, we obtain the formula 
\begin{equation} \label{relation}
    \begin{split}
        DS_\psi f(\textbf{u},b,a)&=\frac{|a|}{(2 \pi)^{n/2}} \int_{\mathbb{R}}{Rf_{\textbf{u}}(p) \overline{\psi(a(p-b))} e^{-iap} dp }=\frac{1}{(2 \pi)^{\frac{n-1}{2}}}S_\psi(Rf_{\textbf{u}})(b,a). 
    \end{split}
\end{equation}
Thus, the directional Stockwell transform is precisely the application of a one-dimensional Stockwell transform to the slices of the
Radon transform where $\textbf{u}$ remains fixed and $p$ varies. The formula (\ref{relation}) presents the relation between the Radon, directional Stockwell and Stockwell transforms.

\begin{definition}
Let $\psi\in \mathcal{S}(\mathbb{R})$ be a non-trivial window. A function $\eta \in \mathcal{S}(\mathbb{R}) $ is a reconstruction window of $\psi$ if 
\begin{equation} \label{5}
C_{\psi,\eta}=\frac{1}{\pi} \int_{\mathbb{R}}{\overline{\widehat{\psi}(\xi-1)} \widehat{\eta}(\xi-1) \frac{d\xi}{|\xi|^n}}
\end{equation}
is nonzero and finite. If $\eta=\psi$ in (\ref{5}) then $\psi$ is called an admissible window.
\end{definition}
One can show that every non-trivial $\psi\in \mathcal{S}(\mathbb{R})$ has a reconstruction window $\eta$ which may be chosen such that $e^{ix}\eta(x) \in \mathcal{S}_0(\mathbb{R})$.

The next proposition gives an extended Parseval's identity.

\begin{proposition}(Extended Parseval's Identity)
Let $\psi\in \mathcal{S}(\mathbb{R})$ be a non-trivial window and let $\eta \in \mathcal{S}(\mathbb{R}) $ be a reconstruction window for it. Then 
$$\int_{\mathbb{R}^n} f(\normalfont\textbf{x}) \overline{h(\normalfont\textbf{x})}d\normalfont\textbf{x}=\frac{1}{C_{\psi,\eta}} \int_{\mathbb{S}^{n-1}}\int_{\mathbb{R}^{\times}} \int_{\mathbb{R}} DS_\psi f(\textbf{u},b,a) \overline{DS_{\eta} h(\textbf{u},b,a)}|a|^{n-2} dbdad\textbf{u},$$
for $f,h \in L^1(\mathbb{R}^n)\cap L^2(\mathbb{R}^n)$.
\end{proposition}
\begin{proof} First we can see that $\widehat{f}(\cdot \textbf{u}) \overline{\widehat{\psi} \left( \frac{\cdot}{a}-1 \right)}, \,\widehat{h}(\cdot \textbf{u}) \overline{\widehat{\eta} \left( \frac{\cdot}{a}-1 \right)} \in L^1(\mathbb{R})\cap L^2(\mathbb{R})$.
Using Proposition 3.2, the Parseval's identity for the Fourier transform and the Fubini's theorem, we obtain
\begin{align*}
&\quad\,\int_{\mathbb{S}^{n-1}}\int_{\mathbb{R}^{\times}} \int_{\mathbb{R}} DS_\psi f(\textbf{u},b,a) \overline{DS_{\eta} h(\textbf{u},b,a)}|a|^{n-2} dbdad\textbf{u} \\
&=\frac{1}{(2 \pi)^{n+2}}\int_{\mathbb{S}^{n-1}}\int_{\mathbb{R}^{\times}} |a|^{n-2}  \left( \int_{\mathbb{R}} \mathcal{F} \left( \widehat{f}(\cdot \textbf{u}) \overline{\widehat{\psi} \left( \frac{\cdot}{a}-1 \right)} \right)(-b) \overline{\mathcal{F} \left( \widehat{h}(\cdot \textbf{u}) \overline{\widehat{\eta} \left( \frac{\cdot}{a}-1 \right)} \right)(-b)} db\right )dad\textbf{u}\\
&=\frac{1}{(2 \pi)^{n+2}}\int_{\mathbb{S}^{n-1}}\int_{\mathbb{R}^{\times}} |a|^{n-2}\left (  \int_{\mathbb{R}} \mathcal{F} \left( \widehat{f}(\cdot \textbf{u}) \overline{\widehat{\psi} \left( \frac{\cdot}{a}-1 \right)} \right)(b) \overline{\mathcal{F} \left( \widehat{h}(\cdot \textbf{u}) \overline{\widehat{\eta} \left( \frac{\cdot}{a}-1 \right)} \right)(b)} db\right ) dad\textbf{u}\\
&=\frac{1}{(2 \pi)^{n+1}} \int_{\mathbb{S}^{n-1}}\int_{\mathbb{R}^{\times}} |a|^{n-2} \left (\int_{\mathbb{R}}  \widehat{f}(\xi \textbf{u}) \overline{\widehat{h}(\xi \textbf{u})} \widehat{\eta} \left( \frac{\xi}{a}-1 \right) \overline{\widehat{\psi} \left( \frac{\xi}{a}-1 \right)} d \xi\right ) dad\textbf{u} \\
&=\frac{1}{(2 \pi)^{n+1}} \int_{\mathbb{R}} d \xi \int_{\mathbb{S}^{n-1}} \widehat{f}(\xi \textbf{u}) \overline{\widehat{h}(\xi \textbf{u})} d\textbf{u} \int_{\mathbb{R}^{\times}} |a|^{n-2} \widehat{\eta} \left( \frac{\xi}{a}-1 \right) \overline{\widehat{\psi} \left( \frac{\xi}{a}-1 \right)} da\\
&=\frac{1}{(2 \pi)^{n+1}} \int_{\mathbb{R}} |\xi|^{n-1} d \xi \int_{\mathbb{S}^{n-1}} \widehat{f}(\xi \textbf{u}) \overline{\widehat{h}(\xi \textbf{u})} d\textbf{u} \int_{\mathbb{R}}  \widehat{\eta} \left( t-1 \right) \overline{\widehat{\psi} \left( t-1 \right)} \frac{dt}{|t|^n}\end{align*}
\begin{align*}
&= \frac{C_{\psi,\eta}}{2}\frac{1}{(2 \pi)^{n}} \int_{\mathbb{R}} |\xi|^{n-1} d \xi \int_{\mathbb{S}^{n-1}} \widehat{f}(\xi \textbf{u}) \overline{\widehat{h}(\xi \textbf{u})} d\textbf{u}\\
&=\frac{C_{\psi,\eta}}{(2 \pi)^{n}} \int_{0}^{\infty} |\xi|^{n-1} d \xi \int_{\mathbb{S}^{n-1}} \widehat{f}(\xi \textbf{u}) \overline{\widehat{h}(\xi \textbf{u})} d\textbf{u}\\&=\frac{C_{\psi,\eta}}{(2 \pi)^{n}} \int_{\mathbb{R}^n} \widehat{f}(\textbf{w}) \overline{\widehat{h}(\textbf{w})} d\textbf{w}=C_{\psi,\eta} \int_{\mathbb{R}^n} f(\textbf{x}) \overline{h(\textbf{x})}d\textbf{x}.
\end{align*} \end{proof}
The next proposition gives the reconstruction formula.

\begin{proposition}(Reconstruction formula)
Let $\psi\in \mathcal{S}(\mathbb{R})$ be a non-trivial window and let $\eta \in \mathcal{S}(\mathbb{R}) $ be a reconstruction window for it. If $f\in L^1(\mathbb{R}^n)$ such that $\widehat{f}\in L^1(\mathbb{R}^n) $, then 
$$f(\normalfont\textbf{x})=\frac{1}{C_{\psi,\eta}} \int_{\mathbb{S}^{n-1}}\int_{\mathbb{R}^{\times}} \int_{\mathbb{R}} DS_\psi f(\textbf{u},b,a) \eta _{\textbf{u},b,a}(\textbf{x})|a|^{n-2} dbdad\textbf{u},   \enspace a.e.\,\, \textbf{x}\in \mathbb{R}^n.$$
\end{proposition}
\begin{proof}
Using Proposition 3.2 and the Fubini's theorem twice, we have
\begin{align*}
&\quad\,\int_{\mathbb{S}^{n-1}}\int_{\mathbb{R}^{\times}} \int_{\mathbb{R}} DS_\psi f(\textbf{u},b,a) \eta _{\textbf{u},b,a}(\textbf{x})|a|^{n-2} dbdad\textbf{u}\\
&= \frac{1}{(2 \pi)^{n+1}} \int_{\mathbb{S}^{n-1}}\int_{\mathbb{R}^{\times}} e^{ia(\textbf{u}\cdot \textbf{x})} |a|^{n-1} dad\textbf{u}  \int_{\mathbb{R}} \eta (a(\textbf{u} \cdot \textbf{x} -b)) e^{-iab}db   \int_{\mathbb{R}}{\widehat{f}(\xi \textbf{u}) \overline{\widehat{\psi}\left(\frac{\xi}{a}-1\right)}  e^{i \xi b} d \xi}\\
&=\frac{1}{(2 \pi)^{n+1}} \int_{\mathbb{S}^{n-1}}\int_{\mathbb{R}^{\times}} e^{ia(\textbf{u}\cdot \textbf{x})} |a|^{n-1} dad\textbf{u}  \int_{\mathbb{R}}{\widehat{f}(\xi \textbf{u}) \overline{\widehat{\psi}\left(\frac{\xi}{a}-1\right)}  d \xi}  \int_{\mathbb{R}} \eta(a(\textbf{u} \cdot \textbf{x} -b)) e^{i(\xi -a)b}db\\
&=\frac{1}{(2 \pi)^{n+1}} \int_{\mathbb{S}^{n-1}}\int_{\mathbb{R}^{\times}} e^{ia(\textbf{u}\cdot \textbf{x})} |a|^{n-1} dad\textbf{u}  \int_{\mathbb{R}}{\widehat{f}(\xi \textbf{u}) \overline{\widehat{\psi}\left(\frac{\xi}{a}-1\right)}  d \xi}  \int_{\mathbb{R}} \frac{1}{|a|} \eta(t) e^{i(\xi -a)\textbf{u}\cdot \textbf{x}} e^{-i(\frac{\xi}{a} -1)t}dt\\
&=\frac{1}{(2 \pi)^{n+1}} \int_{\mathbb{S}^{n-1}}\int_{\mathbb{R}^{\times}} |a|^{n-2} dad\textbf{u} \int_{\mathbb{R}}  e^{i \xi(\textbf{u}\cdot \textbf{x})} \widehat{f}(\xi \textbf{u}) \overline{\widehat{\psi}\left(\frac{\xi}{a}-1\right) } \widehat{\eta}\left(\frac{\xi}{a}-1\right) d \xi \\
&=\frac{1}{(2 \pi)^{n+1}} \int_{\mathbb{R}} d \xi \int_{\mathbb{S}^{n-1}} e^{i \xi(\textbf{u}\cdot \textbf{x})} \widehat{f}(\xi \textbf{u}) d\textbf{u}  \int_{\mathbb{R}^{\times}} |a|^{n-2} \overline{\widehat{\psi}\left(\frac{\xi}{a}-1\right) } \widehat{\eta}\left(\frac{\xi}{a}-1\right) da \\
&=\frac{1}{(2 \pi)^{n+1}} \int_{\mathbb{R}} |\xi|^{n-1} d \xi \int_{\mathbb{S}^{n-1}}  \widehat{f}(\xi \textbf{u}) e^{i (\xi \textbf{u})\cdot \textbf{x}}d\textbf{u}  \int_{\mathbb{R}} \overline{\widehat{\psi}\left(t-1\right) } \widehat{\eta}\left(t-1\right) \frac{dt}{|t|^n}\\
&=\frac{C_{\psi,\eta}}{2(2 \pi)^{n}} \int_{\mathbb{R}} |\xi|^{n-1} d \xi \int_{\mathbb{S}^{n-1}}  \widehat{f}(\xi \textbf{u}) e^{i (\xi \textbf{u})\cdot \textbf{x}}d\textbf{u}\\&=\frac{C_{\psi,\eta}}{(2 \pi)^{n}} \int_{0}^{\infty} |\xi|^{n-1} d \xi \int_{\mathbb{S}^{n-1}}  \widehat{f}(\xi \textbf{u}) e^{i (\xi \textbf{u})\cdot \textbf{x}}d\textbf{u}\\
&= \frac{C_{\psi,\eta}}{(2 \pi)^{n}} \int_{\mathbb{R}^n}{\widehat{f}(\textbf{w})e^{i\textbf{w}\cdot \textbf{x}}}d\textbf{w}=C_{\psi,\eta} f(\textbf{x}),
\end{align*}
for a.e. $ \textbf{x}\in \mathbb{R}^n$.
\end{proof}

According to our choice of the standard measure on $\mathbb{Y}^{n+1}$, we denote by $L^2(\mathbb{Y}^{n+1}):=L^2(\mathbb{Y}^{n+1},|a|^{n-2}dbdad\textbf{u} )$. So, the inner product on this space is
$$( F,G )_{L^2(\mathbb{Y}^{n+1})}:=\int_{\mathbb{S}^{n-1}}\int_{\mathbb{R}^{\times}} \int_{\mathbb{R}} F(\textbf{u},b,a) \overline{G(\textbf{u},b,a)} |a|^{n-2} dbdad\textbf{u}.$$
Now, under the conditions of Proposition 3.4, we have
\begin{equation*}( f,h)_{L^2(\mathbb{R}^{n})}= \frac{1}{C_{\psi,\eta}}( DS_\psi f,DS_{\eta} h )_{L^2(\mathbb{Y}^{n+1})}.\end{equation*}
If $\psi$ is an admissible window and $f\in L^1(\mathbb{R}^n)\cap L^2(\mathbb{R}^n) $ then 
$$\| DS_\psi f\|_{L^2(\mathbb{Y}^{n+1})}=\sqrt{C_{\psi,\psi}}\,\,\| \,f\|_{L^2(\mathbb{R}^{n})}.$$
Since $\mathcal{S}(\mathbb{R}^n)$ is a dense subset of $L^2(\mathbb{R}^n)$ and the last relation holds true for all $f \in \mathcal{S}(\mathbb{R}^n) $, $DS_{\psi}$ can be extended to a constant multiple of an isometric embedding $L^2(\mathbb{R}^n)\to L^2(\mathbb{Y}^{n+1})$.

The reconstruction formula suggests us to define %\textit{\textbf{the directional Stockwell synthesis operator}} 
the directional Stockwell synthesis operator that maps the function on $\mathbb{Y}^{n+1}$ to a function on $\mathbb{R}^n$.  
Given $\psi\in \mathcal{S}(\mathbb{R})$, we define the directional Stockwell synthesis operator as
$$DS^{*}_{\psi} \Phi(\textbf{x}):=\int_{\mathbb{S}^{n-1}}\int_{\mathbb{R}^{\times}} \int_{\mathbb{R}} \Phi(\textbf{u},b,a) \psi_{\textbf{u},b,a}(\textbf{x}) |a|^{n-2} dbdad\textbf{u}, \enspace \textbf{x}\in \mathbb{R}^n.$$
The last integral is absolutely convergent if $\Phi\in \mathcal{S}(\mathbb{Y}^{n+1})$.
Now, by using the reconstruction formula, we obtain
$$DS^{*}_{\eta}(DS_\psi f)(\textbf{x}):=\int_{\mathbb{S}^{n-1}}\int_{\mathbb{R}^{\times}} \int_{\mathbb{R}} DS_\psi f(\textbf{u},b,a) \eta_{\textbf{u},b,a}\textbf{(x}) |a|^{n-2} dbdad\textbf{u}=C_{\psi,\eta}f(\textbf{x}),$$
for a.e. $\textbf{x}\in \mathbb{R}^n$.

We also show that the directional Stockwell synthesis operator is in fact the transpose of the directional Stockwell transform in the following sense:
\begin{proposition}
Let $\psi\in \mathcal{S}(\mathbb{R})$. If $f\in L^1(\mathbb{R}^n)$ and $\Phi\in \mathcal{S}(\mathbb{Y}^{n+1})$, then
$$\int_{\mathbb{R}^n}f(\normalfont\textbf{x}){\overline{DS_{\psi}^{*}{\overline\Phi}}}(\textbf{x})d\textbf{x}=\int_{\mathbb{S}^{n-1}}\int_{\mathbb{R}^{\times}} \int_{\mathbb{R}} DS_\psi f(\textbf{u},b,a) \Phi(\textbf{u},b,a) |a|^{n-2} dbdad\textbf{u}.$$
\end{proposition}
\begin{proof}
Using the Fubini's theorem, we have
\begin{align*}
\int_{\mathbb{R}^n}f(\textbf{x}){\overline{DS_{\psi}^{*}{\overline\Phi}}}(\textbf{x})d\textbf{x} &=\int_{\mathbb{R}^n}f(\textbf{x})d\textbf{x} \int_{\mathbb{S}^{n-1}}\int_{\mathbb{R}^{\times}} \int_{\mathbb{R}} \Phi(\textbf{u},b,a) \overline{\psi_{\textbf{u},b,a}(\textbf{x})}  |a|^{n-2} dbdad\textbf{u}\\
&=\int_{\mathbb{S}^{n-1}}\int_{\mathbb{R}^{\times}} \int_{\mathbb{R}} \Phi(\textbf{u},b,a) |a|^{n-2} dbdad\textbf{u} \int_{\mathbb{R}^n}f(\textbf{x})\overline{\psi_{\textbf{u},b,a}(\textbf{x})} d\textbf{x}\\
&=\int_{\mathbb{S}^{n-1}}\int_{\mathbb{R}^{\times}} \int_{\mathbb{R}} DS_\psi f(\textbf{u},b,a) \Phi(\textbf{u},b,a) |a|^{n-2} dbdad\textbf{u}.
\end{align*}
\end{proof}
Under the standard identification (\ref{standardidentification}), we write the last relation as
$$\big\langle f, \overline{DS_{\psi}^{*}{\overline\Phi}} \,\big\rangle = \langle  DS_{\psi}f, \Phi \rangle.$$
The last relation will be our model for defining the distributional directional Stockwell transform in Section 5.

%%%%%%%%%%%%%%%%%%%%%%%%%%%%%%%%%%%%%%%%%%%%%%%%%%%
\section{Continuity of the directional Stockwell transform on test function spaces}

Let $\mathcal{S}_1(\mathbb{R})=\lbrace \varphi \in \mathcal{S}(\mathbb{R}): e^{ix}\varphi(x) \in \mathcal{S}_0(\mathbb{R}) \rbrace$ and $\psi \in \mathcal{S}_1(\mathbb{R})$.
In the next two theorems we show the continuity of $DS_\psi:\mathcal{S}_0(\mathbb{R}^n)  \to \mathcal{S}(\mathbb{Y}^{n+1})$ and $DS_{\psi}^{*}: \mathcal{S}(\mathbb{Y}^{n+1}) \to \mathcal{S}_0(\mathbb{R}^n)$.

Notice that we can extend the definition of the directional Stockwell transform as a sesquilinear mapping $DS:(f,\psi)\to DS_{\psi}f $, whereas the directional Stockwell synthesis operator extends to the bilinear form 
$DS^{*}:(\Phi,\psi)\to DS_{\psi}^{*}\Phi $.

\begin{theorem}
The bilinear mapping 
$DS:\mathcal{S}_0(\mathbb{R}^n) \times \mathcal{S}_1(\mathbb{R}) \to \mathcal{S}(\mathbb{Y}^{n+1})$ defined as
$\left( f,\psi \right) \to DS_\psi f$
is continuous.
\end{theorem}

\begin{proof}
To prove the theorem is enough to show that for given $l,m,k,r,s\in \mathbb{N}_0$, there exist $\nu,\tau \in \mathbb{N}_0$ and $C>0$ such that
\begin{equation}\label{th4.1,10}\rho_{s,r}^{l,m,k}(DS_\psi f) \leq C\rho_{\nu}(f) \rho_{\tau}(\psi),  \enspace f\in \mathcal{S}_0(\mathbb{R}^n),\ \psi\in \mathcal{S}_1(\mathbb{R}).\end{equation}
We may suppose that $r$ is even and $s\geq 1.$\\
\textbf{Step 1.} By the definition of the directional Stockwell transform and the Leibniz’s formula, we obtain
    \begin{align*}
    &\,\quad \left|\frac{\partial^l}{\partial a^l}\frac{\partial^m}{\partial b^m} DS_\psi f(\textbf{u},b,a)\right|\\
    &=\frac{1}{(2 \pi)^{n/2}} \left| \frac{\partial^l}{\partial a^l}\frac{\partial^m}{\partial b^m} \int_{\mathbb{R}^n} af(\textbf{x}) \overline{\psi(a(\textbf{x} \cdot \textbf{u}-b))} e^{-ia(\textbf{x} \cdot \textbf{u})} d\textbf{x} \right| \\
    &=\frac{1}{(2 \pi)^{n/2}} \left| \frac{\partial^l}{\partial a^l} \int_{\mathbb{R}^n} a^{m+1} f(\textbf{x}) \overline{\psi^{(m)}(a(\textbf{x} \cdot \textbf{u}-b))} e^{-ia(\textbf{x} \cdot \textbf{u})} d\textbf{x} \right| \\
    &=\frac{1}{(2 \pi)^{n/2}} \bigg| \int_{\mathbb{R}^n}  f(\textbf{x}) \sum_{k_1+k_2+k_3=l} \binom{l}{k_1,k_2,k_3} \binom{m+1}{k_1} a^{m+1-k_1} \\
    & \times e^{-ia(\textbf{x} \cdot \textbf{u})} (-i(\textbf{u} \cdot \textbf{x}))^{k_2} \overline{\psi^{(m+k_3)}(a(\textbf{x} \cdot \textbf{u}-b))} (\textbf{u} \cdot \textbf{x}-b)^{k_3}  d\textbf{x} \bigg| \\
&\lesssim \frac{1}{(2 \pi)^{n/2}} \sum_{k_1+k_2+k_3=l} \sum_{|\alpha|=k_2} \sum_{|\beta|\leq k_3} \bigg| \int_{\mathbb{R}^n} f(\textbf{x}) a^{m+1-k_1} e^{-ia(\textbf{x} \cdot \textbf{u})} \textbf{x}^{\alpha} \textbf{u}^{\alpha} \end{align*}
\begin{align*}
& \times\overline{\psi^{(m+k_3)}(a(\textbf{x} \cdot \textbf{u}-b))} \textbf{x}^{\beta} \textbf{u}^{\beta} b^{k_3-|\beta|}  d\textbf{x} \bigg|\\
&\leq \frac{1}{(2 \pi)^{n/2}} \sum_{k_1+k_2+k_3=l} \sum_{|\alpha|=k_2} \sum_{|\beta|\leq k_3} \left( |a|^{m}+|a|^{-m} \right) \left( 1+|b|^2 \right)^{\frac{k_3-|\beta|}{2}}  \\
& \times\bigg|  \int_{\mathbb{R}^n} a f(\textbf{x})  e^{-ia(\textbf{x} \cdot \textbf{u})} \textbf{x}^{\alpha+\beta}  \overline{\psi^{(m+k_3)}(a(\textbf{x} \cdot \textbf{u}-b))} d\textbf{x} \bigg|\\
& \leq \left( |a|^{m}+|a|^{-m} \right) \left( 1+|b|^2 \right)^{l/2} \sum_{k_1+k_2+k_3=l} \sum_{|\alpha|=k_2} \sum_{|\beta|\leq k_3} \left| DS_{\psi_{m+k_3}} f_{\alpha+\beta}(\textbf{u},b,a) \right|,
\end{align*}
where $\psi_{m+k_3}(x)=\psi^{(m+k_3)}(x) \in \mathcal{S}_1 (\mathbb{R}),$ and $f_{\alpha+\beta}(\textbf{x})=\textbf{x}^{\alpha+\beta}f(\textbf{x}) \in \mathcal{S}_0 (\mathbb{R}^n) $. Then
$$\rho_{s,r}^{l,m,k}(DS_\psi f) \lesssim \sum_{k_1+k_2+k_3=l} \sum_{|\alpha|=k_2} \sum_{|\beta|\leq k_3} {\rho_{m+s,l+r}^{0,0,k}(DS_{\psi_{m+k_3}} f_{\alpha+\beta})} .$$
So we can assume that $m=l=0$.\\
\textbf{Step 2.} First, we have
\begin{align*}
   &\quad\, \left| \Delta_{\textbf{u}}^{k} DS_\psi f(\textbf{u},b,a) \right| =\frac{|a|}{(2 \pi)^{n/2}} \left| 
\Delta_{\textbf{u}}^{k} \left( \int_{\mathbb{R}^n} f(\textbf{x}) \overline{\psi(a(\textbf{x} \cdot \textbf{u}-b))} e^{-ia(\textbf{x} \cdot \textbf{u})} d\textbf{x}  \right) \right| \\
&=\frac{|a|}{(2 \pi)^{n/2}} \left| \sum_{|\alpha|,j,d \leq 2k} P_{\alpha,j,d}(\textbf{u}) \int_{\mathbb{R}^n} a^d \textbf{x}^{\alpha} f(\textbf{x}) \overline{\psi^{(j)}(a(\textbf{x} \cdot \textbf{u}-b))} e^{-ia(\textbf{x} \cdot \textbf{u})} d\textbf{x} \right| \\
&\lesssim \left( |a|^{2k}+|a|^{-2k}\right) \sum_{|\alpha|,j \leq 2k} \frac{|a|}{(2 \pi)^{n/2}} \left| \int_{\mathbb{R}^n}  \textbf{x}^{\alpha} f(\textbf{x}) \overline{\psi^{(j)}(a(\textbf{x} \cdot \textbf{u}-b))} e^{-ia(\textbf{x} \cdot \textbf{u})} d\textbf{x} \right|,
\end{align*}
where the $P_{\alpha,j,d}(\textbf{u})$ are certain polynomials.
If we put $f_{\alpha}(\textbf{x})=\textbf{x}^{\alpha}f(\textbf{x}), \psi_j (x)=\psi^{(j)}(x)$, we obtain 
$$  \left| \Delta_{\textbf{u}}^{k} DS_\psi f(\textbf{u},b,a) \right| \lesssim \left( |a|^{2k}+|a|^{-2k}\right) \sum_{|\alpha|,j \leq 2k} \left| DS_{\psi_j} f_{\alpha} (\textbf{u},b,a) \right|,$$
that yields
$$\rho_{s,r}^{0,0,k}(DS_\psi f) \lesssim \sum_{|\alpha|,j \leq 2k} {\rho_{2k+s,r}^{0,0,0}(DS_{\psi_{j}} f_{\alpha})}.$$
Since $\psi_j \in \mathcal{S}_1 (\mathbb{R}), f_{\alpha} \in \mathcal{S}_0 (\mathbb{R}^n) $, we can suppose that $k=0$.\\
\textbf{Step 3.} Using Proposition 3.2, we have 
\begin{align*}
    &\,\quad\left( 1+|b|^2 \right)^{r/2} \left| DS_\psi f(\textbf{u},b,a)\right|\\
    &=\left( 1+|b|^2 \right)^{r/2} \left| \frac{1}{(2 \pi)^{n/2+1}} \int_{\mathbb{R}} \widehat{f}(\xi \textbf{u}) \overline{\widehat{\psi} \left( \frac{\xi}{a}-1\right)}  e^{i(\xi-a)b} d \xi \right| \end{align*}
\begin{align*}
    &= \frac{1}{(2 \pi)^{n/2+1}}\left|  \int_{\mathbb{R}} \widehat{f}(\xi \textbf{u}) \overline{\widehat{\psi} \left( \frac{\xi}{a}-1\right)}  \left( 1-\frac{\partial^2}{\partial \xi ^2} \right)^{r/2} \left( e^{i(\xi-a)b} \right) d \xi  \right| \\
    &=  \frac{1}{(2 \pi)^{n/2+1}}\left| \int_{\mathbb{R}} \left( 1-\frac{\partial^2}{\partial \xi ^2} \right)^{r/2} \left( \widehat{f}(\xi \textbf{u}) \overline{\widehat{\psi} \left( \frac{\xi}{a}-1\right)} \right) e^{i(\xi-a)b} d \xi  \right| \\
    & =  \frac{1}{(2 \pi)^{n/2+1}}\left| \sum_{|\alpha|,j \leq r} a^{-j} Q_{\alpha,j}(\textbf{u}) \int_{\mathbb{R}} e^{i(\xi-a)b} \widehat{f}^{(\alpha)}(\xi \textbf{u}) \overline{\widehat{\psi}^{(j)} \left( \frac{\xi}{a}-1\right)}   d \xi  \right| \\
    & \lesssim \left( |a|^{r}+|a|^{-r} \right) \sum_{|\alpha|,j \leq r} \frac{1}{(2 \pi)^{n/2+1}}  \left| \int_{\mathbb{R}} e^{i(\xi-a)b} \widehat{\left(\textbf{x}^{\alpha} f(\textbf{x})\right)}  (\xi \textbf{u}) \overline{\widehat{\left( x^j \psi(x)\right)} \left( \frac{\xi}{a}-1 \right)} d \xi \right|,
     \end{align*}
for some polynomials $Q_{\alpha,j}(\textbf{u})$.
If we put $f_{\alpha}(\textbf{x})=\textbf{x}^{\alpha}f(\textbf{x}), \psi_j(x)=x^j \psi(x)$, then by Proposition 3.2 we obtain
\begin{align*}
    &\quad\,\left( 1+|b|^2 \right)^{r/2} \left| DS_\psi f(\textbf{u},b,a)\right|\\
    & \lesssim \left( |a|^{r}+|a|^{-r} \right) \sum_{|\alpha|,j \leq r} \frac{1}{(2 \pi)^{n/2+1}}  \left| \int_{\mathbb{R}} e^{i(\xi-a)b} \widehat{f_{\alpha} }(\xi \textbf{u}) \overline{\widehat{\psi_j} \left( \frac{\xi}{a}-1 \right)} d \xi \right|\\
    &=\left( |a|^{r}+|a|^{-r} \right) \sum_{|\alpha|,j \leq r} \left| DS_{\psi_j} f_{\alpha}(\textbf{u},b,a)\right|.
\end{align*}
So we conclude that
$$\rho_{s,r}^{0,0,0}(DS_\psi f) \lesssim \sum_{|\alpha|,j \leq r} {\rho_{r+s,0}^{0,0,0}(DS_{\psi_{j}} f_{\alpha})}.$$
Since $\psi_j \in \mathcal{S}_1 (\mathbb{R}), f_{\alpha} \in \mathcal{S}_0 (\mathbb{R}^n)$ we can suppose that $r=0$.\\
\textbf{Step 4.}
Since $f\in \mathcal{S}_0 (\mathbb{R}^n) $, using the Taylor expansion
of $\hat f$, we have 
\begin{align*}
   &\,\quad \frac{1}{|a|^s} \left| DS_\psi f(\textbf{u},b,a) \right|= \frac{1}{|a|^s (2 \pi)^{n/2+1}} \left| \int_{\mathbb{R}} e^{i(\xi-a)b} \widehat{f }(\xi \textbf{u}) \overline{\widehat{\psi} \left( \frac{\xi}{a}-1 \right)} d \xi \right| \\
    & =\frac{1}{|a|^{s-1} (2 \pi)^{n/2+1}} \left| \int_{\mathbb{R}} e^{ia(\xi-1)b} \widehat{f }(a \xi \textbf{u}) \overline{\widehat{\psi} \left( \xi-1 \right)} d \xi \right| \end{align*}
\begin{align*}
    & \leq \frac{1}{|a|^{s-1}} \sum_{|\alpha|=s-1} \frac{1}{(2 \pi)^{n/2+1}} \left| \int_{\mathbb{R}} e^{ia(\xi-1)b} \frac{(a \xi \textbf{u})^{\alpha}}{\alpha !}\widehat{f}^{(\alpha)}(a \xi_0 \textbf{u}) \overline{\widehat{\psi} \left( \xi-1 \right)} d \xi \right| \\
    & \lesssim \sum_{|\alpha|=s-1} \left| \int_{\mathbb{R}} e^{ia(\xi-1)b} 
\xi^{s-1} \widehat{f}^{(\alpha)}(a \xi_0 \textbf{u}) \overline{\widehat{\psi} \left( \xi-1 \right)} d \xi \right| \\
& \lesssim  \sum_{|\alpha|=s-1} \int_{\mathbb{R}^n} \left| \textbf{x}^{\alpha} f(\textbf{x}) \right| d\textbf{x} \int_{\mathbb{R}} \left| \xi \right|^{s-1} \left| \widehat{\psi}(\xi-1) \right| d \xi \lesssim \rho_{s+n}(f) \rho_{s+1}(\psi).
\end{align*}
Since $\psi\in \mathcal{S}_1 (\mathbb{R}) $, we have
\begin{align*}
   &\,\quad |a|^s \left| DS_\psi f(\textbf{u},b,a) \right|=\frac{|a|^s}{ (2 \pi)^{n/2+1}} \left| \int_{\mathbb{R}} e^{i(\xi-a)b} \widehat{f }(\xi \textbf{u}) \overline{\widehat{\psi} \left( \frac{\xi}{a}-1 \right)} d \xi \right| \\
    &=\frac{|a|^s}{ (2 \pi)^{n/2+1}} \left| \int_{\mathbb{R}} e^{i(\xi-a)b} \widehat{f }(\xi \textbf{u}) \left( \frac{\xi}{a} \right)^s \frac{1}{s!}\overline{\widehat{\psi}^{(s)} \left( \frac{\xi_0}{a}-1 \right)} d \xi \right| \\
    & \lesssim \int_{\mathbb{R}} \left| \eta^s \psi(\eta) \right| d\eta \int_{\mathbb{R}} \left| \xi \right|^{s} \left| \widehat{f}(\xi \textbf{u}) \right| d \xi \lesssim \rho_{s+n+1}(f) \rho_{s+2}(\psi).
\end{align*}
Then, we obtain
$$\rho_{s,0}^{0,0,0}(DS_\psi f) \lesssim \rho_{s+n+1}(f) \rho_{s+2}(\psi).$$

Summing up all the estimates, we find that (\ref{th4.1,10}) holds true for $\nu=4k+4l+3r+m+s+n+1$ and $\tau=4k+4l+3r+2m+s+2$.
\end{proof}

We now analyze the directional Stockwell synthesis operator.

\begin{theorem}
The bilinear mapping 
$DS^{*}: \mathcal{S}(\mathbb{Y}^{n+1}) \times \mathcal{S}_1(\mathbb{R}) \to \mathcal{S}_0(\mathbb{R}^n)$ defined as 
$(\Phi,\psi) \to DS_{\psi}^{*}\Phi$
is continuous.
\end{theorem}
\begin{proof}
Since $\Phi \in \mathcal{S}(\mathbb{Y}^{n+1}) $  and $\psi \in \mathcal{S}_1(\mathbb{R}) $ then $DS_{\psi}^{*}\Phi \in \mathbb{C}^{\infty}(\mathbb{R}^n).$ 
Using the Parseval's identity and the Fubini's theorem, we can write
\begin{align*}
h(\textbf{x})&:=DS_{\psi}^{*}\Phi(\textbf{x})=\int_{\mathbb{S}^{n-1}}\int_{\mathbb{R}^{\times}} \int_{\mathbb{R}} \Phi(\textbf{u},b,a) \psi_{\textbf{u},b,a}(\textbf{x}) |a|^{n-2} dbdad\textbf{u} \\
    &=\frac{1}{2\pi} \int_{\mathbb{S}^{n-1}} d\textbf{u} \int_{\mathbb{R}^{\times}} |a|^{n-2} da \int_{\mathbb{R}} \widehat{\Phi}(\textbf{u},w,a) \overline{{\widehat{\left( \overline{\psi_{\textbf{u},\cdot,a}(\textbf{x})} \right)}}  (w)} dw\\
    &=\frac{1}{(2\pi)^{n/2+1}} \int_{\mathbb{S}^{n-1}} d\textbf{u} \int_{\mathbb{R}^{\times}} |a|^{n-2} da \int_{\mathbb{R}} \widehat{\Phi}(\textbf{u},w,a) e^{i(a+w)\textbf{x} \cdot \textbf{u}} \widehat{\psi}\left( \frac{w}{a} \right) dw\\
    &=\frac{1}{(2\pi)^{n/2+1}} \int_{\mathbb{S}^{n-1}} d\textbf{u} \int_{\mathbb{R}^{\times}} |a|^{n-2} da \int_{\mathbb{R}} \widehat{\Phi}(\textbf{u},w-a,a) e^{iw(\textbf{x}\cdot \textbf{u}) } \widehat{\psi}\left( \frac{w}{a} -1 \right) dw\\
    &=\frac{1}{(2\pi)^{n/2+1}} \int_{\mathbb{R}} dw \int_{\mathbb{S}^{n-1}} e^{iw(\textbf{x}\cdot \textbf{u}) } d\textbf{u} \int_{\mathbb{R}^{\times}} |a|^{n-2}   \widehat{\Phi}(\textbf{u},w-a,a)  \widehat{\psi}\left( \frac{w}{a} -1 \right) da\\
    &=\frac{1}{(2\pi)^{n/2+1}} \int_{0}^{\infty} w^{n-1} dw \int_{\mathbb{S}^{n-1}} e^{iw(\textbf{x}\cdot \textbf{u}) } d\textbf{u} \int_{\mathbb{R}} \bigg(  \widehat{\Phi}(\textbf{u},w-\frac{w}{t},\frac{w}{t}) \\
    &+ \widehat{\Phi}(-\textbf{u},-w+\frac{w}{t},-\frac{w}{t}) \bigg) \widehat{\psi}\left( t -1 \right)  \frac{dt}{|t|^n},
\end{align*}
 where $\widehat{\Phi}$ stands for the Fourier transform of $\Phi(\textbf{u},b,a)$ with respect to the variable $b$.
 
Now we define
$$G(\textbf{u}\textbf{},w):=\int_{\mathbb{R}} \Big(  \widehat{\Phi}(\textbf{u},w-\frac{w}{t},\frac{w}{t}) + \widehat{\Phi}(-\textbf{u},-w+\frac{w}{t},-\frac{w}{t}) \Big) \widehat{\psi}\left( t -1 \right)  \frac{dt}{|t|^n}, $$
$w \in \mathbb{R}, \textbf{u}\in \mathbb{S}^{n-1}.$
Clearly $G$ is a smooth function and
\begin{equation} \label{6}
    h(\textbf{x})=\frac{1}{(2\pi)^{n/2+1}} \int_{0}^{\infty} w^{n-1} dw \int_{\mathbb{S}^{n-1}} G(\textbf{u},w) e^{iw(\textbf{x}\cdot \textbf{u}) } d\textbf{u}.
\end{equation}
We will show that $G\in \mathcal{S}(\mathbb{S}^{n-1} \times \mathbb{R})$. For given $N,q,k \in \mathbb{N}_0$, it holds
\begin{align*}
    &\,\quad\sup_{(\textbf{u},w)\in \mathbb{S}^{n-1} \times \mathbb{R}} \left| w^N \frac{\partial^q}{\partial w^q} \Delta_\textbf{u}^{k} \left( \int_{\mathbb{R}}  \widehat{\Phi}(\textbf{u},w-\frac{w}{t},\frac{w}{t}) \widehat{\psi}\left( t -1 \right)  \frac{dt}{|t|^n}\right) \right|\\
    &=\sup_{(\textbf{u},w)\in \mathbb{S}^{n-1} \times \mathbb{R}} \left| w^N    \int_{\mathbb{R}}  \frac{\partial^q}{\partial w^q} \Delta_\textbf{u} ^{k} \widehat{\Phi}(\textbf{u},w-\frac{w}{t},\frac{w}{t}) \widehat{\psi}\left( t -1 \right)  \frac{dt}{|t|^n} \right|\\
    &=\sup_{(\textbf{u},w)\in \mathbb{S}^{n-1} \times \mathbb{R}} \left | w^N  \int_{\mathbb{R}} \widehat{\psi}(t-1) \frac{dt}{|t|^n} \int_{\mathbb{R}} \sum_{l_1+l_2=q} \binom{q}{l_1,l_2} \frac{\partial^{l_1}}{\partial w ^{l_1}} \left(  \Delta_\textbf{u} ^{k} \Phi(\textbf{u},b,\frac{w}{t})\right) \right .\\
    &\quad\times\left . \frac{\partial^{l_2}}{\partial w ^{l_2}} \left( e^{-ib(w-\frac{w}{t})} \right)
  db  \right| \\
    &=\sup_{(\textbf{u},w)\in \mathbb{S}^{n-1} \times \mathbb{R}} \left | w^N      \sum_{l_1+l_2=q} \binom{q}{l_1,l_2} \int_{\mathbb{R}} \left( 1-\frac{1}{t} \right) ^{l_2} \frac{1}{t^{l_1}} \widehat{\psi}(t-1) \frac{dt}{|t|^n}\right.\\
    &\quad \int_{\mathbb{R}} 
 b^{l_2} \left.\frac{\partial^{l_1} \Delta_\textbf{u} ^{k} \Phi(\textbf{u},b,\frac{w}{t})}{\partial \left( \frac{w}{t} \right) ^{l_1}}    e^{-iw(1-\frac{1}{t})b} db  \right| \lesssim \rho_{r'}(\widehat{\psi}) \rho_{v_1',v_2'}^{q_1',q_2',k}(\Phi) ,
\end{align*}
for some $q_1',q_2',r',v_1',v_2' \in \mathbb{N}_0$. Similar, we have 
$$ \sup_{(\textbf{u},w)\in \mathbb{S}^{n-1} \times \mathbb{R}} \left| w^N \frac{\partial^q}{\partial w^q} \Delta_\textbf{u} ^{k} \left( \int_{\mathbb{R}}  \widehat{\Phi}(-\textbf{u},-w+\frac{w}{t},-\frac{w}{t}) \widehat{\psi}\left( t -1 \right)  \frac{dt}{|t|^n}\right) \right| \lesssim \rho_{r''}(\widehat{\psi}) \rho_{v_1'',v_2''}^{q_1'',q_2'',k}(\Phi),$$ 
for some $q_1'',q_2'',r'',v_1'',v_2'' \in \mathbb{N}_0$.
So, we have shown that for given $N,q,k \in \mathbb{N}_0$, 
\begin{equation} \label{***} \sup_{(\textbf{u},w)\in \mathbb{S}^{n-1} \times \mathbb{R}} \left| w^N \frac{\partial^q}{\partial w^q} \Delta_\textbf{u} ^{k} G(\textbf{u},w) \right| \lesssim \rho_{r}(\widehat{\psi}) \rho_{v_1,v_2}^{q_1,q_2,k}(\Phi),\end{equation}
for some $q_1,q_2,r,v_1,v_2 \in \mathbb{N}_0$.
 
Then, by relation (\ref{6}) and the Fourier inversion formula in polar coordinates, we obtain $h\in \mathcal{S}(\mathbb{R}^n)$ and 
\begin{equation}\label{****}\widehat{h}(w\textbf{u})=(2 \pi)^{n/2-1}G(\textbf{u},w),\end{equation}
for $w\in \mathbb{R}^{+}, \textbf{u}\in \mathbb{S}^{n-1} $. 
To show that $h\in \mathcal{S}_0 (\mathbb{R}^n)$ is equivalent of showing that 
$$\lim_{\textbf{w} \to 0} \frac{\widehat{h}(\textbf{w})}{|\textbf{w}|^k}=0,  \enspace \forall k\in \mathbb{N}_0.$$
Since $\Phi\in \mathcal{S}(\mathbb{Y}^{n+1})$, there exists $C_k>0$ such that 
$$|\widehat{\Phi}(\textbf{u},\omega,a)| \leq C_k |a|^{k+1},$$
uniformly for $\omega\in \mathbb{R}, \textbf{u} \in \mathbb{S}^{n-1}.$
Then, for $w\in \mathbb{R}^{+}$ and $\textbf{u}\in \mathbb{S}^{n-1},$  
$$|\widehat{h}(w\textbf{u})| \lesssim C_k \int_{\mathbb{R}} \left| \frac{w}{t} \right|^{k+1} |\widehat{\psi}(t-1)| \frac{dt}{|t|^{n}}. $$

Since  $\psi\in \mathcal{S}_1 (\mathbb{R})$ we have that $\int_{\mathbb{R}} |\widehat{\psi}(t-1)| \frac{dt}{|t|^{n+k+1}} < \infty$. Then 
$\displaystyle\lim_{w\to 0^+} \frac{\widehat{h}(w\textbf{u})}{|w|^k}=0,$ 
 $  \forall k\in \mathbb{N}_0,$
i.e.
$$\lim_{\textbf{w}\to 0} \frac{\widehat{h}(\textbf{w})}{|\textbf{w}|^k}=0, \enspace \forall k\in \mathbb{N}_0.$$
Thus $ h=DS_\psi ^{*} \Phi\in \mathcal{S}_0 (\mathbb{R}^n).$

We now prove the continuity of the bilinear mapping 
$DS^{*}$. Since $\varphi \to \widehat{\varphi}$ is an automorphism of the $\mathcal{S}$
spaces, the family 
$$\widehat{\rho}_v(\varphi)=\rho_v(\widehat{\varphi}), \enspace \varphi\in \mathcal{S} (\mathbb{R}), \ v\in \mathbb{N}_0,$$
is a base of seminorms for the topology of $\mathcal{S} (\mathbb{R})$. We will define different family of seminorms on $\mathcal{S}_0 (\mathbb{R}^n)$. The seminorms $$\dot{\rho}_{N,q,k}(\varphi):=\sup_{(\textbf{u},w)\in \mathbb{S}^{n-1} \times \mathbb{R}} \left| w^N \frac{\partial^q}{\partial w^q} \Delta_\textbf{u} ^{k} \widehat{\varphi}(w\textbf{u}) \right|, \enspace N,q,k \in \mathbb{N}_0, $$
are a base of continuous seminorms for the topology of $\mathcal{S}_0 (\mathbb{R}^n)$ (\cite{Kos13}, p. 10). So, for given $N,q,k \in \mathbb{N}_0$, by equations (\ref{***}) and (\ref{****}) we obtain
\begin{align*}
    &\dot{\rho}_{N,q,k}(h)=\sup_{(\textbf{u},w)\in \mathbb{S}^{n-1} \times \mathbb{R}} \left| w^N \frac{\partial^q}{\partial w^q} \Delta_\textbf{u} ^{k} \widehat{h}(w\textbf{u}) \right|\\ & \lesssim \sup_{(\textbf{u},w)\in \mathbb{S}^{n-1} \times \mathbb{R}} \left| w^N \frac{\partial^q}{\partial w^q} \Delta_\textbf{u} ^{k} G(\textbf{u},w) \right| \lesssim \rho_{r}(\widehat{\psi}) \rho_{v_1,v_2}^{q_1,q_2,k}(\Phi),
\end{align*}
for some $q_1,q_2,r,v_1,v_2 \in \mathbb{N}_0.$
\end{proof}
The next proposition is a direct consequence of Proposition 3.5, Theorem 4.1 and Theorem 4.2.

\begin{proposition}
Let $\psi\in \mathcal{S}_1 (\mathbb{R})$ be a non-trivial window and $\eta\in \mathcal{S}_1 (\mathbb{R})$ be a reconstruction window for it. Then
\begin{equation} \label{reconsutrctionformula}
\frac{1}{C_{\psi,\eta}}DS_{\eta}^{*} \circ DS_\psi= Id_{\mathcal{S}_0 (\mathbb{R}^n)}.
\end{equation} 
\end{proposition}

%%%%%%%%%%%%%%%%%%%%%%%%%%%%%%%%%%%%%%%%%%%%%%%%%%%%%%%%%%%%%%%%%

\section{Directional Stockwell transform on $\mathcal{S}_0'(\mathbb{R}^n)$}

We are ready to define the directional Stockwell transform of Lizorkin distributions.
\begin{definition} \label{5.3}
Let $\psi\in \mathcal{S}_1 (\mathbb{R})$. We define the directional Stockwell transform of $f\in \mathcal{S}_0' (\mathbb{R}^n)$ with respect to $\psi$ as the element $DS_\psi f \in \mathcal{S}'(\mathbb{Y}^{n+1})$ whose action on test functions is given by 
$$\langle DS_\psi f,\Phi\rangle:= \big\langle f, \overline{DS_\psi ^{*}\overline{\Phi}}\,\big\rangle, \enspace \Phi\in \mathcal{S}(\mathbb{Y}^{n+1}).$$
\end{definition}

\begin{definition} \label{5.4}
Let $\psi\in \mathcal{S}_1 (\mathbb{R})$. We define the directional Stockwell synthesis operator 
$DS_\psi ^{*}:\mathcal{S}'(\mathbb{Y}^{n+1}) \to \mathcal{S}_0' (\mathbb{R}^n)  $ as
$$\langle DS_\psi ^{*} F,\varphi \rangle:=\big\langle F,\overline{DS_\psi \overline{\varphi}} \,\big\rangle, \enspace F \in \mathcal{S}'(\mathbb{Y}^{n+1}), \enspace \varphi \in \mathcal{S}_0 (\mathbb{R}^n). $$
\end{definition}

The Definitions (\ref{5.3}) and (\ref{5.4}) are valid according to Theorem 4.2 and Theorem 4.1, respectively.

The following continuity result is obtained by taking transposes in Theorems 4.1 and Theorem 4.2. 

\begin{proposition}
Let $\psi \in \mathcal{S}_1 (\mathbb{R})$. The directional Stockwell transform $DS_{\psi}:\mathcal{S}_0' (\mathbb{R}^n) \to \mathcal{S}'(\mathbb{Y}^{n+1}) $ and the directional Stockwell synthesis operator $DS_\psi ^{*}:\mathcal{S}'(\mathbb{Y}^{n+1}) \to \mathcal{S}_0' (\mathbb{R}^n)  $ are continuous linear maps.
\end{proposition}
We can also generalize the reconstruction formula (\ref{reconsutrctionformula}) to the space of Lizorkin distributions.
\begin{proposition}
Let $\psi\in \mathcal{S}_1 (\mathbb{R})$ be a non-trivial window and $\eta \in \mathcal{S}_1 (\mathbb{R})$ be a reconstruction window for it. Then 
$$\frac{1}{C_{\psi,\eta}}DS_{\eta}^{*} \circ DS_\psi= Id_{\mathcal{S}_0'(\mathbb{R}^n)}.$$
\end{proposition}
\begin{proof}
Let $f\in \mathcal{S}_0'(\mathbb{R}^n)$ and $\varphi \in \mathcal{S}_0(\mathbb{R}^n)$. Using Definitions \ref{5.3} and \ref{5.4} and Proposition 4.3, we obtain 
$$\langle DS_{\eta}^{*} \circ DS_\psi f,\varphi \rangle=\big\langle DS_\psi f,\overline{DS_{\eta} \overline{\varphi}} \,\big\rangle=\big\langle f,\overline{DS_\psi ^{*}(DS_{\eta}\overline{\varphi})} \,\big\rangle=\langle f, \overline{C_{\eta,\psi}\overline{\varphi}}\rangle=C_{\psi,\eta}\langle f,\varphi \rangle.$$
\end{proof}

%%%%%%%%%%%%%%%%%%%%%%%%%%%%%%%%%%%%%%%%%%%%%%%%%%%%%%%%%%%

\section{Directional Stockwell transform on $D_{L^1}'(\mathbb{R}^n)$}

Since, $\psi_{\textbf{u},b,a}\in D_{L^{\infty}}(\mathbb{R}^n) $, Definition 3.1 also works for $f\in D_{L^1}'(\mathbb{R}^n) $. In this case, by (\ref{dlp}) and (\ref{def_DST}) we have 
\begin{equation} \label{8'}
DS_\psi f(\textbf{u},b,a):= \sum_{j=1}^{N} \langle \partial^{\alpha_j} f_{j}(\textbf{x}), \overline{\psi_{\textbf{u},b,a}(\textbf{x})} \rangle_{\textbf{x}}= \sum_{j=1}^{N} (-1)^{|\alpha_j|} \int_{\mathbb{R}^n} f_j(\textbf{x}) \overline{\frac{\partial^{|\alpha_j|}}{\partial \textbf{x}^{\alpha_j}} \psi_{\textbf{u},b,a}(\textbf{x})} d\textbf{x},
\end{equation}
where $f_j \in L^1(\mathbb{R}^n)$.
 
Since the space $D_{L^1}'(\mathbb{R}^n)$ contains the space of compactly supported distributions $\mathcal{E}'(\mathbb{R}^n)$ 
and the space of convolutors $\mathcal{O}'_C(\mathbb{R}^n)$,  Definition 3.1 also makes sense for $f\in \mathcal{E}'(\mathbb{R}^n) $ or $f\in\mathcal{O}'_C(\mathbb{R}^n)$. On the other hand, Definition 3.1 does not work for $f\in \mathcal{S}'(\mathbb{R}^n) $ since $\psi_{\textbf{u},b,a} \notin \mathcal{S}(\mathbb{R}^n)$.

In the next Proposition we show that Definition 3.1 coincides with Definition 5.1 under our convention (\ref{standardidentification}) for identifying functions with distributions on $\mathbb{Y}^{n+1}$.

\begin{proposition}
Let $f\in D_{L^1}'(\mathbb{R}^n) $ and $\psi\in \mathcal{S}_1(\mathbb{R})$. The directional Stockwell transform of $f$ is given by the function (\ref{def_DST}), i.e.
$$\langle DS_\psi f, \Phi \rangle= \int_{\mathbb{S}^{n-1}}\int_{\mathbb{R}^{\times}} \int_{\mathbb{R}} DS_\psi f({\normalfont\textbf{u}},b,a) \Phi(\normalfont\textbf{u},b,a) |a|^{n-2} dbdad\normalfont\textbf{u}, \enspace \Phi \in  \mathcal{S}(\mathbb{Y}^{n+1}).$$
\end{proposition}
\begin{proof}
Since $f\in D_{L^1}'(\mathbb{R}^n) $, by (\ref{dlp}) we have that $f=\sum_{j=1}^{N} \partial^{\alpha_j} f_{j}$, where $f_j \in L^1(\mathbb{R}^n)$
for some $N\in \mathbb{N}, \alpha_j \in \mathbb{N}_0 ^{n}$. Now for $\Phi \in \mathcal{S}(\mathbb{Y}^{n+1})$, we obtain
\begin{equation} \label{8}
\langle DS_\psi(\partial^{\alpha_j}f_j) , \Phi \rangle:=\big\langle \partial^{\alpha_j}f_j,\overline{DS_\psi ^{*}\overline{\Phi}} \, \big\rangle = (-1)^{|\alpha_j|} \Big \langle f_j (\textbf{x}), \frac{\partial^{|\alpha_j|}}{\partial \textbf{x}^{\alpha_j}}\overline{DS_\psi ^{*}\overline{\Phi}(\textbf{x})} \,\Big \rangle.
\end{equation}
On the other hand,
\begin{align*} 
\frac{\partial^{|\alpha_j|}}{\partial \textbf{x}^{\alpha_j}}\overline{DS_\psi ^{*}\overline{\Phi}(\textbf{x})}&=\frac{\partial^{|\alpha_j|}}{\partial \textbf{x}^{\alpha_j}} \overline{\int_{\mathbb{S}^{n-1}}\int_{\mathbb{R}^{\times}} \int_{\mathbb{R}} \overline{\Phi}(\textbf{u},b,a) \psi_{\textbf{u},b,a}(\textbf{x}) |a|^{n-2} dbdad\textbf{u}}
\end{align*}

\begin{align}\label{9}
&\nonumber =\int_{\mathbb{S}^{n-1}}\int_{\mathbb{R}^{\times}} \int_{\mathbb{R}} {\Phi}(\textbf{u},b,a) \frac{|a|}{(2\pi)^{n/2}} \frac{\partial^{|\alpha_j|}}{\partial \textbf{x}^{\alpha_j}}\Big (\overline{\psi}\big(a(\textbf{x} \cdot \textbf{u}-b)\big) e^{-ia(\textbf{x} \cdot \textbf{u})}\Big ) |a|^{n-2} dbdad\textbf{u}\\
&\nonumber =\sum_{k \leq \alpha_j} \binom{\alpha_j}{k} (-i)^{|\alpha_j-k|} \overline{\int_{\mathbb{S}^{n-1}}\int_{\mathbb{R}^{\times}} \int_{\mathbb{R}} (a\textbf{u})^{\alpha_j}\overline{\Phi}(\textbf{u},b,a) \psi_{\textbf{u},b,a}^{(|k|)}(\textbf{x}) |a|^{n-2} dbdad\textbf{u}}\\
&=\sum_{k \leq \alpha_j} \binom{\alpha_j}{k} (-i)^{|\alpha_j-k|} \overline{DS_{\psi^{(|k|)}} ^{*}((a\textbf{u})^{\alpha_j}\overline{\Phi})(\textbf{x})}.
\end{align}

By relations (\ref{8}) and (\ref{9}), we obtain

\begin{align*}
\langle DS_\psi(\partial^{\alpha_j}f_j), \Phi \rangle&=(-1)^{|\alpha_j|} \sum_{k \leq \alpha_j} \binom{\alpha_j}{k} (-i)^{|\alpha_j-k|}\Big \langle f_j (\textbf{x}), \overline{DS_{\psi^{(|k|)}} ^{*}((a\textbf{u})^{\alpha_j}\overline{\Phi})(\textbf{x})}\,\Big\rangle\\
&=(-1)^{|\alpha_j|} \sum_{k \leq \alpha_j} \binom{\alpha_j}{k} (-i)^{|\alpha_j-k|} \big\langle DS_{\psi^{(|k|)}} f_j, (a\textbf{u})^{\alpha_j}\Phi \big\rangle.
\end{align*}

The last relation implies 
$$DS_\psi(\partial^{\alpha_j}f_j)=(-1)^{|\alpha_j|} \sum_{k \leq \alpha_j} \binom{\alpha_j}{k} (-i)^{|\alpha_j-k|} (a\textbf{u})^{\alpha_j} DS_{\psi^{(|k|)}} f_j.$$
Therefore, we can suppose that $f\in L^1(\mathbb{R}^n)$.\\ 
Now, under this assumption, Proposition 3.6 and the standard identification (\ref{standardidentification}), we have 

\begin{align*}
&\quad\,\langle DS_\psi f, \Phi \rangle= \big\langle f, \overline{DS_\psi^{*}\overline{\Phi}}\,\big \rangle= \int_{\mathbb{R}^n} f(\textbf{x}) \overline{DS_\psi^{*}\overline{\Phi}(\textbf{x})} d\textbf{x}\\
&=\int_{\mathbb{S}^{n-1}}\int_{\mathbb{R}^{\times}} \int_{\mathbb{R}} DS_\psi f(\textbf{u},b,a) \Phi (\textbf{u},b,a) |a|^{n-2} dbdad\textbf{u}=\langle DS_\psi f(\textbf{u},b,a), \Phi(\textbf{u},b,a)\rangle.
\end{align*}
\end{proof}

%%%%%%%%%%%%%%%%%%%%%%%%%%%%%%%%%%%%%

 %%%%%%%%%%%%%%%%%%%%%%%%%%%%%%%

\end{document}